\pdfoutput=1
\documentclass{amsart}

\usepackage{ MnSymbol, graphicx, verbatim, hyperref}
\newcommand{\pg}{\Sigma}
\newcommand{\mon}{\varphi}

\newcommand{\obd}{(\Sigma,\varphi)}

\newcommand{\B}{\mathcal{B}}

\newcommand{\C}{\mathcal{C}}

\newcommand{\triple}{(\Sigma,\varphi,\Gamma)}

\newcommand{\intint}{\mathring{\cap}}

\newtheorem*{thm_1}{Theorem \ref{thm_tight}}
\newtheorem*{thm_2}{Theorem \ref{thm_surgery}}
\newtheorem{thm}{Theorem}[section]
\newtheorem{cor}[thm]{Corollary}
\newtheorem{lem}[thm]{Lemma}

\theoremstyle{definition}
\newtheorem{defin}[thm]{Definition}

\theoremstyle{remark}

\newtheorem{obs}[thm]{Observation}

\begin{document}

\title[Tightness is preserved by Legendrian surgery]{Tightness is preserved by Legendrian surgery}
\author{Andy Wand}

\begin{abstract}
This paper describes a characterization of tightness of closed contact 3-manifolds in terms of supporting open book decompositions. The main result is that tightness of a closed contact 3-manifold is preserved under Legendrian surgery.

\end{abstract}

\maketitle
\section{Introduction}
Following the developments of Donaldson theory \cite{d} (and later the Seiberg-Witten equations and Taubes's Gromov invariants - see e.g. \cite{t}), smooth low-dimensional topology has become increasingly intertwined with complex and symplectic geometry. Along the way, the study of contact structures on 3-manifolds has been brought to a place of prominence, as these keep track of a complex/symplectic structure near a boundary component of a 4-manifold, allowing one to port the cut-and-paste tools of smooth 4-dimensional topology into the complex/symplectic categories. Contact structures split into two types: \emph{tight} and \emph{overtwisted}. A fundamental theorem of Eliashberg \cite{el} gives a complete classification of overtwisted structures, in particular showing that each homotopy class of plane fields contains a unique isotopy class of overtwisted contact structures. As such, if a contact structure is to carry any geometric information, it must be tight. The classification of tight structures however has remained largely open.

This paper describes a characterization of tightness of closed contact 3-manifolds in terms of supporting open book decompositions. In particular, we introduce the notion of \emph{consistency} of a mapping class of an oriented surface with boundary (equivalently, of an open book decomposition of a closed 3-manifold), and show:

\begin{thm}\label{thm_tight}
Let $M$ be a closed, oriented 3-manifold, and $\xi$ a positive co-oriented contact structure on $M$. Then the following are equivalent:

\begin{enumerate}

\item $\xi$ is tight.

\item Some open book decomposition supporting $(M,\xi)$ is consistent.

\item Each open book decomposition supporting $(M,\xi)$ is consistent.

\end{enumerate} 

\end{thm}

In light of the correspondence theorem of Giroux \cite{gi}, Theorem \ref{thm_tight} reduces the study of tightness to the study of surface diffeomorphisms. It should be noted that, as our aim in this paper is mainly a specific application, we will restrict ourselves to a `stable' description of consistency. In contrast to the general theory developed in \cite{w2} (see \cite{w3} for an expository account), the stable version does not aim to give tools to check consistency of an arbitrary open book decomposition.

A symplectic (or Stein) 4-manifold comes equipped with an almost complex structure $J$ (i.e. an endomorphism of the tangent bundle whose restriction to each tangent space squares to $-Id$); as such, any boundary component $M$ comes with an induced plane field $\xi = TM  \cap JTM$. We say such a pair $(M,\xi)$ is \emph{(symplectically/Stein) fillable}.  While fillable contact structures are tight (by work of Gromov \cite{gr} and Eliashberg \cite{el3}), the converse is not necessarily true.

Again reaching to the analogy between the smooth and symplectic/Stein categories, a central notion in each is that of a handle attachment; for a 4-dimensional manifold, by far the most interesting case is that of a 2-handle attachment. It was shown by Eliashberg \cite{el2} and Weinstein \cite{we} that, if a 2-handle is attached to a symplectic/Stein 4-manifold  along a curve in the contact boundary everywhere tangent to the contact structure, with a framing coefficient one less than that determined by the contact structure, we may extend the symplectic/Stein structure in a unique way over the handle. The trace of this operation on the contact boundary is referred to as a \emph{Legendrian surgery}, and is a fundamental tool for constructing fillable contact manifolds. While of course these are all tight, this in itself does little to advance our understanding of non-fillable tight structures, as the relation between tightness and Legendrian surgery was little understood. Indeed the only known result in this direction, due to Honda \cite{ho2}, was an example of an \emph{open} tight contact manifold which becomes overtwisted through Legendrian surgery. Our main result then fills in this gap, showing that:

\begin{thm}\label{thm_surgery}
If $(M,\xi)$ is obtained by Legendrian surgery on tight $(M',\xi')$, for $M'$ a closed, oriented 3-manifold, and $\xi'$ a positive co-oriented contact structure, then $(M,\xi)$ is tight. 
\end{thm}

Section \ref{sec_def} introduces terminology and definitions, while Section \ref{sec_properties} gathers some properties of consistency, showing in particular that it is both determined by any basis of arcs in the surface, and also invariant under stabilization (in the sense of Giroux), thus determines a property of the supported contact structure.  Section \ref{sec_proofs} is then devoted to the proofs of Theorems \ref{thm_tight} and \ref{thm_surgery}.

\subsection*{Acknowledgments}
We would like to thank Paolo Ghiggini, Patrick Massot, Vera V\'ertesi, and the referee for helpful comments and suggestions, and the Max Planck Institut f\"ur Mathematik, Harvard University, and Universit\'e de Nantes for support and hospitality. This work was partially supported by ERC Grant GEODYCON.

\section{Definitions}\label{sec_def}
\subsection{Preliminaries}
Throughout the paper, $M$ will refer to a closed, smooth, oriented 3-manifold, and $\xi$ will denote a (positive) co-oriented contact structure on $M$; i.e. $\xi$ is the kernel of some globally defined 1-form $\alpha$ on $M$, such that $\alpha \wedge d\alpha$ is a (positive) volume form for $M$. One says $\xi$ is \emph{overtwisted} if there is some embedded disc $D$ in $M$ such that the tangent plane of each point $p\in \partial D$ agrees with $\xi_p$; otherwise $\xi$ is \emph{tight}. An \emph{open book decomposition} for $M$ is a pair consisting of an embedded oriented link $B$ in $M$, as well as a fibration of the complement of $B$ over $S^{1}$, such that each fiber is the interior of a Seifert surface for $B$. We encode this structure as the pair $\obd$, where $\pg$, the \emph{page}, is the oriented Seifert surface, and $\mon$, the \emph{monodromy}, is the return map of the fibration. We consider $\mon$ as an element of $\pi_0Diff^+(\pg,\partial)$, the \emph{mapping class group} of $\pg$, which we will refer to simply as $MCG(\pg)$. The pair $\obd$ determines the open book decomposition up to a diffeomorphism of $M$, and is often referred to as an \emph{abstract} open book (see e.g. \cite{e}). When there is no fear of confusion we will drop the term `abstract'.

A central notion concerning open book decompositions is that of `stabilization', which corresponds to a plumbing of a Hopf band (see e.g. \cite{gi}). This operation may be encoded in an abstract open book as follows:

\begin{defin}\label{def_stab}
Let $\obd$ be an open book decomposition of $M$, and $\sigma$ a properly embedded arc in $\pg$. Let $\pg'$ denote the surface given by attaching a 1-handle to $\pg$ with attaching sphere $\partial \sigma$, and denote by $s \subset \pg'$ the simple closed curve gotten by taking the union of $\sigma$ with the core of the new handle. Then the pair $(\pg',\tau_s \circ \mon)$, where $\tau_s$ denotes the Dehn twist about $s$, and $\mon$ the obvious inclusion of the original $\mon$ (extended over the handle by the identity), is again an open book decomposition of $M$, referred to as a \emph{stabilization} of $\obd$, via $\sigma$.
\end{defin}

The relation between the above concepts is given most completely by the following `correspondence theorem' of Giroux:

\begin{thm}\label{thm_giroux}\cite{gi}
Let $M$ be a closed, oriented, smooth 3-manifold. Then there is a 1-1 correspondence between positive co-oriented contact structures on $M$ up to isotopy, and open book decompositions of $M$ up to isotopy and stabilization.
\end{thm}

One says that a contact structure is \emph{supported by} each open book to which it is associated through this correspondence.

\subsection{Overtwisted regions}

Let $\pg$ be a closed surface with boundary. Throughout the paper, an \emph{arc} in $\pg$ will refer to a properly embedded arc. We will refer to a set $\Gamma$ of (oriented) disjoint arcs in $\pg$ as an \emph{(oriented) arc collection}, while an arc collection which cuts $\pg$ into a disc is a \emph{basis} (of $\pg$).

A main object of study in the paper will be `augmented' open books $\triple$, for $\Gamma$ an arc collection. As such, when stabilizing an augmented open book $\triple$, we will isotope $\Gamma$ such that the 1-handle is attached away from $\partial \Gamma$, so that we have an inclusion map $\iota$ of $\Gamma$ into the stabilized book such that the image is again an arc collection.

We will often be interested in \emph{stable} properties of augmented open book decompositions:

\begin{defin}
Let $P$ be some property of augmented open book decompositions. Then, given an augmented open book $\triple$, we say that $\triple$ \emph{stably} satisfies $P$ if there is some sequence of positive stabilizations after which the stabilized triple $(\pg',\mon',\iota(\Gamma))$ satisfies $P$.
\end{defin}
 
By convention, if $\gamma$ is an oriented arc in $\pg$, for an open book decomposition $\obd$, then its image $\mon(\gamma)$ is given the opposite orientation. In particular then for an oriented arc collection $\Gamma$, each $p \in \Gamma \cap \mon(\Gamma)$ may be given a sign, as follows: if the ordered pair in $T_p\pg$ consisting of the tangent vector along the element of $\Gamma$, followed by the tangent vector along the element of $\mon(\Gamma)$, gives the orientation of $\pg$ at $p$, then $p$ is \emph{positive}. Otherwise $p$ is \emph{negative} (Figure \ref{fig_signs}).

\begin{figure}[h!]
\centering \scalebox{1.1}{\includegraphics{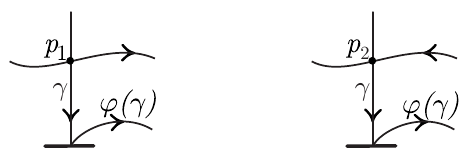}}
\caption[]{The points $p_1$ and $p_2$ are, respectively, positive and negative intersection points. This figure introduces the conventions, which will hold throughout the paper, that elements of a given arc collection are drawn as straight lines, their images under a mapping class are curved, and in any figure with multiple line weights, the thickest lines are reserved for $\partial \pg$.}
\label{fig_signs}
\end{figure}

\begin{defin}\label{def_ot_region}
Let $\obd$ be a open book decomposition, and $\Gamma$ an oriented arc collection in $\pg$ such that each point of $\partial\Gamma$ is positive in $\Gamma \cap \mon(\Gamma)$. An \emph{overtwisted  region} (in $\triple$) is an embedded disc $A \hookrightarrow \pg$, with $\partial A \hookrightarrow (\Gamma \cup \mon(\Gamma))$, such that:

\begin{enumerate}

\item Corners of $A$ alternate between points in $\partial \Gamma$, and negative points in the interior of $\pg$.
\item Each point of $\Gamma \cap \mon(\Gamma) \cap int(\pg)$ is a corner of $A$.
\item $A$ is the unique such disc.
\end{enumerate}
\end{defin}

 \begin{figure}[h!]
\centering \scalebox{1.1}{\includegraphics{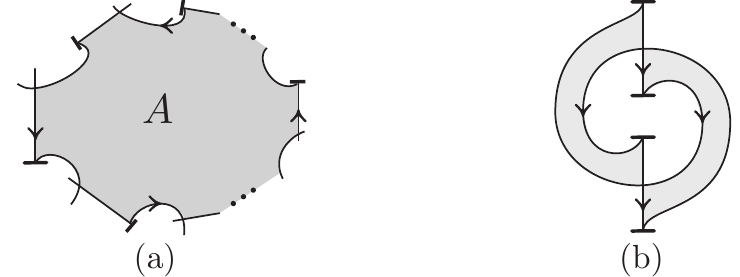}}
\caption[]{(a) An overtwisted region. (b) Each of the illustrated discs satisfy (1) and (2), but not (3).}
\label{fig_OTA}
\end{figure}

\begin{obs}
It should be emphasized that there is no assumption concerning minimality of $\Gamma \cap \mon(\Gamma)$; in particular, for the case $n=1$, an overtwisted region is a bigon.
\end{obs} 

As we shall see, existence of an overtwisted region in $\triple$ implies existence of an overtwisted disc in $(M,\xi)$. Finally:

\begin{defin}
 A class $\mon \in MCG(\pg)$ is \emph{inconsistent} if there is some arc collection $\Gamma$ such that, stably, $\triple$ has an overtwisted region. Otherwise, $\mon$ is \emph{consistent}.
\end{defin}

\section{Properties of consistency}\label{sec_properties}

The purpose of this section is to show firstly that consistency is determined by any basis of arcs, and using this, secondly that consistency is preserved under stabilization and destabilization, and thus is a property of the associated contact structure. Looking forward to the proof of Theorem \ref{thm_surgery}, we will in fact show a bit more. In what follows, a \emph{curve system} in a surface $\pg$ will refer to a collection $L$ of embedded arcs and closed curves in $\pg$, disjoint away from $\partial \pg$, and such that $\partial L \subset \partial \pg$ (in fact, all results in this section hold equally for collections whose components are neither disjoint nor embedded, but as our main application requires consideration of neither of these possibilities we have restricted to this more standard set-up).

\begin{defin}\label{def_proper}
Let $A$ be an overtwisted region in $\triple$, and $L$ a curve system. Then $A$ is \emph{proper with respect to $L$} if there is a negative corner $y$ of $A$ such that for any neighborhood $U$ of $y$, $L$ can be isotoped, fixing $\partial L$, such that $L\cap (\Gamma \cup \mon(\Gamma)) \subset U$.

\end{defin}

\subsection{Basis independence}

\begin{defin}\label{def_detect}
Let $\obd$ be an open book decomposition, $L$ a curve system and $\B$ a basis of $\Sigma$. We say $\B$ \emph{(stably) detects overtwistedness relative to $L$} if there is some arc collection $\Gamma$ such that:
\begin{enumerate}
\item Each element of $\Gamma$ is isotopic to an element of $\B$, and

\item there is a sequence of stabilizations $\triple \leadsto (\pg',S\mon,\iota(\Gamma))$ (where $S$ refers to the composition of the Dehn twists associated to the stabilizations), and a subsequence $S'$ of $S$, such that $(\pg',S\mon,\iota(\Gamma))$ has an overtwisted region $A$, proper with respect to $S'(L)$.
\end{enumerate}

\end{defin}

\begin{obs}It should be emphasized that elements of $\B$ are not assumed oriented, so the isotopy of condition (1) is not an isotopy of oriented arcs. In particular, $\Gamma$ may contain parallel arcs with opposite orientations.
\end{obs}

 Now, existence of such a basis clearly implies inconsistency. Our immediate goal is to show that if some basis of a given open book decomposition stably detects overtwistedness relative to some given curve system, then every basis does (relative to the same curve system). Our main tool for navigating among bases is the following:

\begin{defin}
An \emph{arc-slide domain} (in $\triple$) is a disc component $\Delta$ of $\pg$ cut along $\Gamma$ whose boundary contains exactly three (distinct) elements of $\Gamma$. 
\end{defin}

We pause to gather some notation, to be used throughout the subsection. Firstly, for a given oriented arc $a$, we will denote its endpoints by $\partial^-a$ and $\partial^+a$, such that the orientation points from $\partial^-a$ to $\partial^+a$. Then, let $\obd$ denote an open book decomposition, $\Gamma= \Gamma_A \cup \Gamma_\Delta$ an arc collection, and $L$ a curve system, such that:

\begin{enumerate}
\item $\Gamma_A$ is a minimal collection such that $(\pg,\mon,\Gamma_A)$ has an overtwisted region, which we label $A$. We label $\Gamma_A = \{\gamma_1,\gamma_2,\ldots,\gamma_n\}$ where indices increase around $\partial A$ in accordance with its positive orientation, and label the negative corners of $A$ by $y_i$, such that $y_i \in \gamma_i$. 

\item $A$ is proper with respect to $L$.

\item $(\pg,\mon,\Gamma)$ contains an arc slide domain $\Delta$, with edges $\Gamma_{\Delta} = \{\gamma_1,\gamma_b,\gamma_a\}$, which appear in that order around the boundary of $\Delta$ in the orientation given by $\gamma_1$, and such that $\partial L \cap \Delta = \emptyset$.

\item There is a basis $\B$ of $\pg$ such that each element of $\Gamma_A \cup \{\gamma_a\}$ is isotopic to some element of $\B$.

\end{enumerate}

\begin{figure}[h!]
\centering \scalebox{1.1}{\includegraphics{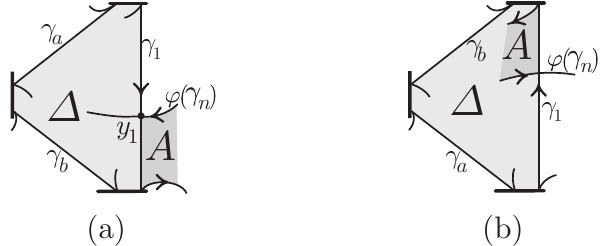}}
\caption[]{}
\label{fig_A}
\end{figure}

The set-up is indicated in Figure \ref{fig_A}(a) for the case that the orientation of $\gamma_1$ disagrees with the positive orientation of $\partial \Delta$, and Figure \ref{fig_A}(b) otherwise. Note that Definition \ref{def_ot_region} allows two possibilities for the orientations of $\Gamma_A$; we choose that which agrees with the positive orientation of $\partial A$. We will also require some terminology to refer to the conditions of Definition \ref{def_ot_region}, as follows: A \emph{region} (in $\triple$) will refer to any embedded disc in $\pg$ with boundary in $\Gamma \cup \mon(\Gamma)$. A region is \emph{boundary based} if exactly every 2nd corner is on $\partial \pg$, and each of these is positive. Finally a region is \emph{isolated} if each point $\Gamma \cap \mon(\Gamma) \cap int(\pg)$ is a corner of the region. In particular, then, an overtwisted region is a unique isolated boundary based region.

We have:
\begin{lem}\label{lem_1} The basis $\B':=(\B \setminus \{\gamma_1\}) \cup \{\gamma_b\}$ stably detects overtwistedness with respect to $L$.
\end{lem}

\begin{proof}
We would like to assume that $\Gamma_A$ contains neither $\gamma_a$ nor $\gamma_b$. As such, observe that, if $\Gamma_A$ contains $\gamma_a$, we may simply replace it in $\Gamma_A$ with a copy pushed slightly out of $\Delta$ (and similarly for $\gamma_b$). We then orient $\gamma_a$ and $\gamma_b$ to disagree with the orientation given to $\partial\Delta$ by $\gamma_1$.

We will begin with a simplification of our data. In particular, let $\sigma_1$ denote an arc in the isotopy class of $\gamma_b$, isotoped in a neighborhood of $\Delta$ to intersect each of $\gamma_a$ and $\gamma_b$ exactly once, and consider the stabilization via $\sigma_1$. As in Definition \ref{def_stab}, we denote the closed curve obtained as the union of $\sigma_1$ with the core of the new handle by $s_1$. Observe then that (using the symbol $\intint$ to denote intersections away from $\partial\pg$), $\gamma_b \intint (\tau_{s_1}\mon(\Gamma) \cup \tau_{s_1}(L))$ is a single point $y \in\tau_{s_1}\mon(\gamma_a)$, where $y$ is an endpoint of a $\partial \pg$-parallel component of $\tau_{s_1}\mon(\gamma_a) \cap \Delta$ with other endpoint $\partial^+\gamma_a$ (Figure \ref{fig_A_stab}(a)). Now, $\sigma_1 \cap \Gamma_A = \emptyset$, and $\tau_{s_1}(L)$ is isotopic to $L$ in a neighborhood of $A$, so we may for notational simplicity assume $\mon$ is a composition with the stabilizing twist $\tau_{s_1}$, i.e. we relabel, setting $\mon := \tau_{s_1}\mon$, and $L:=\tau_{s_1}(L)$.

Similarly, letting $\sigma_2$ denote an arc in the isotopy class of $\mon(\gamma_a)$, isotoped in a neighborhood of $\mon(\Delta)$ to intersect each of $\mon(\gamma_a)$ and $\mon(\gamma_b)$ exactly once, $(\Gamma \cup L) \intint\tau_{s_2}\mon(\gamma_a)$ is again the single point $y$ (Figure \ref{fig_A_stab}(b)). We again re-label, setting $\mon := \tau_{s_2}\mon$ (and not changing $L$).

\begin{figure}[h!]
\centering \scalebox{1.1}{\includegraphics{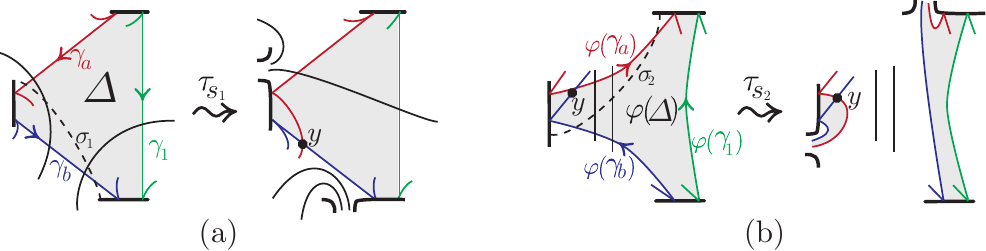}}
\caption[]{(a) The stabilizing arc $\sigma_1$, and the result of the stabilization. (b) The stabilizing arc $\sigma_2$, and the result of the stabilization.}
\label{fig_A_stab}
\end{figure}

We go through the remainder of the proof under the additional condition that the orientation of $\gamma_1$ does not agree with the positive orientation of $\partial\Delta$, then indicate the changes necessary for the remaining case.

Suppose firstly that $A$ is a bigon (Figure \ref{fig_bigon}(a)). After the above simplification then, each of $\gamma_a$ and $\gamma_b$ is mapped into $\Delta$ from its positive endpoint, exiting through the other. In particular, $\{\gamma_a,\gamma_b\}$ determines an isolated boundary based 4-gon region $A' \subset \Delta$ (Figure \ref{fig_bigon}(b)), proper with respect to $L$. Supposing then that $A'$ were not the unique such region, any other is an incident 4-gon with positive corners $\partial^-\gamma_a$ and $\partial^-\gamma_b$, and in particular can have no intersection with $\mon(\gamma_1)$, which is then isotopic to $\gamma_1$, giving an incident bigon $B$ for $A$ (Figure \ref{fig_bigon}(c)), a contradiction. 

\begin{figure}[h!]
\centering \scalebox{1.1}{\includegraphics{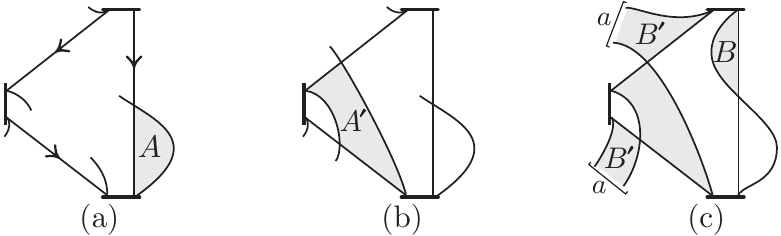}}
\caption[]{(a) The bigon $A$. (b) The 4-gon $A'$. (c) Incident regions $B'$ for $A'$, and $B$ for $A$. This figure introduces the convention, followed through the remainder of the proof, that brackets with like symbol are identified.}
\label{fig_bigon}
\end{figure}

For the case that $A$ is not a bigon (Figure \ref{fig_assumptions}(a)), we have a similar argument, but a bit more to keep track of. Now, after the stabilization via $\sigma_1$, $\gamma_b \cap (\mon(\Gamma_A)\cup L) = \emptyset$, so any intersection of $\Delta$ with $\mon(\Gamma_A)\cup L$ is an arc connecting $\gamma_1$ to $\gamma_a$. By isolation of $A$, $\Delta \cap \mon(\Gamma_A)$ is exactly two arcs: one along $\mon(\gamma_1)$, with endpoint $\partial^-\gamma_1$, the other along $\mon(\gamma_n)$, with endpoint $y_1$. These arcs then are edges of a 4-gon $X \subset \Delta$, such that $(\Gamma_\Delta \intint \mon(\Gamma_A)) \subset \{$corners of $X\}$ (Figure \ref{fig_assumptions}(b)).

Similarly, after the stabilization via $\sigma_2$, $\mon(\gamma_a) \cap (\Gamma_A \cup L) = \emptyset$, so any intersection of $\mon(\Delta)$ with $\Gamma_A \cup L$ is an arc connecting $\mon(\gamma_1)$ to $\mon(\gamma_b)$. There are again exactly two arcs in $\mon(\Delta) \cap \Gamma_A$, one along $\gamma_2$ with endpoint $y_2$, the other along $\gamma_1$ with endpoint $\partial^+ \gamma_1$ (Figure \ref{fig_assumptions}(b)). Again, each is an edge of a 4-gon $Z \subset \mon(\Delta)$, such that $(\Gamma_A \intint \mon(\Gamma_\Delta)) \subset \{$corners of $Z\}$.

Combining the above, we have 
\begin{align*}
\Gamma \intint \mon(\Gamma) &=  (\Gamma_A \intint \mon(\Gamma_A)) \cup (\Gamma_\Delta \intint \mon(\Gamma_A)) \cup (\Gamma_A \intint \mon(\Gamma_\Delta)) \cup (\{\gamma_a,\gamma_b\} \intint \mon(\{\gamma_a,\gamma_b\})) \\
&\subset \{ \textrm{corners of }A\} \cup \{ \textrm{corners of }X\} \cup \{ \textrm{corners of }Z\}  \cup \{y\}.
\end{align*} 

In particular, then $(\Gamma \setminus \{\gamma_1\}) \intint \mon(\Gamma  \setminus \{\gamma_1\})$ are the negative corners of an isolated boundary based region $A'$, constructed by removing $Z$ from $A$, and extending the result over $\gamma_1$ into $\Delta$ (Figure \ref{fig_assumptions}(c)). Moreover, if $L$ can be isotoped such that $L \cap (\Gamma_A \cup \mon(\Gamma_A))$ lies in a neighborhood of $y_1$, then it may also be isotoped to intersect $(\Gamma \setminus \{\gamma_1\}) \cup \mon(\Gamma \setminus \{\gamma_1\})$ in a neighborhood of the corner of $A'$ interior to $\gamma_a$. In particular then, $A'$ is proper with respect to $L$. Moreover, if $B'$ were any other region, then the region $B$ obtained by removing $\mon(\Delta) \setminus Z$ from $B'$ and extending the result over $X$ would be a region in the original data, contradicting our assumption that $A$ be an overtwisted region. We conclude that $A'$ is overtwisted.

\begin{figure}[h!]
\centering \scalebox{.9}{\includegraphics{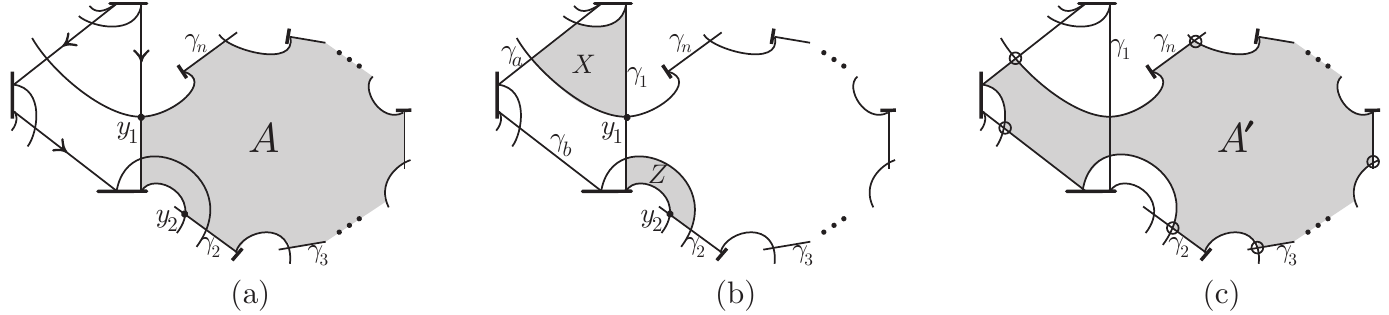}}
\caption[]{The regions $A$ and $A'$.}
\label{fig_assumptions}
\end{figure}

It remains then to consider the case that the orientation of $\gamma_1$ agrees with the positive orientation of $\partial \Delta$. We simply observe that the above proof goes through unchanged, but with the roles of $\Delta$ and $\mon(\Delta)$ reversed. Thus, in the case that $A$ is a bigon (Figure \ref{fig_dual_1}(a)), $A'$ is a 4-gon in $\mon(\Delta)$, while in the general case $A'$ is now obtained from $A$ by removing a 4-gon of $A\cap \Delta$, and extending into $\mon(\Delta)$ (Figure \ref{fig_dual_1}(b)).

\begin{figure}[h!]
\centering \scalebox{1.1}{\includegraphics{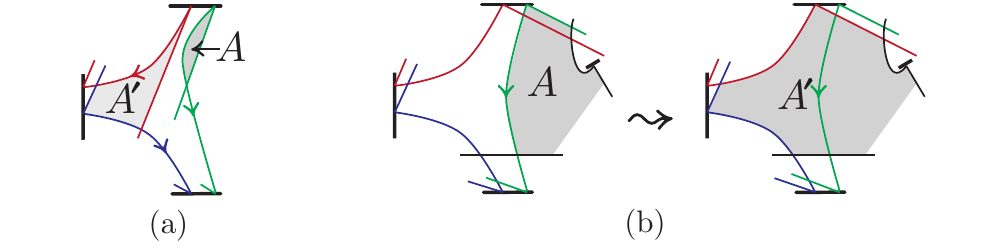}}
\caption[]{}
\label{fig_dual_1}
\end{figure}

\end{proof}

\begin{cor}\label{cor_basis}
Let $\obd$ be an open book decomposition, and $\pg$ admit a basis which stably detects overtwistedness relative to some curve system $L$. Then each basis of $\pg$ stably detects overtwistedness relative to $L$.
\end{cor}

\begin{proof}
Let $\B,\Gamma,S$ and $S'$ be as in Definition \ref{def_detect}.  Throughout the proof, given a stabilization sequence $S$, we will use $\iota_S$ to denote the inclusion associated to $S$. Suppose then that $\B'$ is another basis, obtained from $\B$ by arc slide $\{\gamma_1,\gamma_a\} \leadsto \{\gamma_a,\gamma_b\}$. We enlarge the domain $\Delta$ of the slide to contain each element of $\Gamma_1$ isotopic to $\gamma_1$, and label the collection of such arcs $\{\gamma_1^i\}, i=1,2\ldots,m$, such that $\gamma_1^j$ lies in the 4-gon component of $\Delta$ cut along $\gamma_1^k$ if and only if $k<j$ (Figure \ref{fig_arcslide}(a)).

Let $\gamma_a^1$ denote an arc in the isotopy class of $\gamma_a$, isotoped within $\Delta$ such that no component of either $\partial L$ or of the attaching sphere of any stabilization handle from the sequence $S$ lies between $\gamma_a^1$ and $\gamma^1_1$ in $\partial \pg \cap \partial \Delta$ (Figure \ref{fig_arcslide}(b)). There is then an arc $\gamma_b^1$ in the isotopy class of $\gamma_b$, such that $\iota_S(\gamma_a^1),\iota_S(\gamma_b^1)$, and $\iota_S(\gamma_1^1)$ are edges of an arc-slide domain $\Delta^1 \subset \iota_S(\Delta)$. 

Let $\B_1$ be a basis of the stabilized surface containing $\iota_S(\Gamma)$ and $\B_1'$ the result of the arc-slide $\{\iota_S(\gamma_1^1),\iota_S(\gamma_a^1)\} \leadsto \{\iota_S(\gamma_a^1),\iota_S(\gamma_b^1)\}$  (Figure \ref{fig_arcslide}(c)). Then, by Lemma \ref{lem_1}, there exists a stabilization sequence $S^1$ (which is of course actually just a pair of stabilizations) such that $\iota_{S^1S}((\Gamma \setminus \{\gamma_1^1\}) \cup \{\gamma_a^1,\gamma_b^1\})$ determines an overtwisted region, proper with respect to $S'(L)$ (for some subsequence $S'$ of $S^1S$). We may then proceed in the obvious way: for each $2\leq i\leq m$, letting $\gamma_a^i$ denote an arc in the isotopy class of $\gamma_a$, such that $\iota_{S^{i-1} \cdots S^2S^1S}(\gamma_a^i)$ is adjacent to  $\iota_{S^{i-1} \cdots S^2S^1S}(\gamma_1^i)$, we define $\Gamma_1^i := (\Gamma^{i-1}_1 \setminus \{\gamma_1^i\}) \cup \{\gamma_a^i,\gamma_b^i\}$ (where $\Gamma^0_1 := \Gamma$). Again, by Lemma \ref{lem_1}, there exists a stabilization sequence $S^i$ such that  $\iota_{S^i \cdots S^2S^1S}(\Gamma^i_1)$ has an overtwisted region proper with respect to $S'(L)$, for $S'$ a subsequence of $S^i \cdots S^2S^1S$. Moreover, each element of $\Gamma_1^m$ (Figure \ref{fig_arcslide}(d)) is isotopic to one of $\B'$. As any two bases of $\pg$ are related by such arc-slides (see e.g. \cite{hkm2}), we are done.

\begin{figure}[h!]
\centering \scalebox{1.1}{\includegraphics{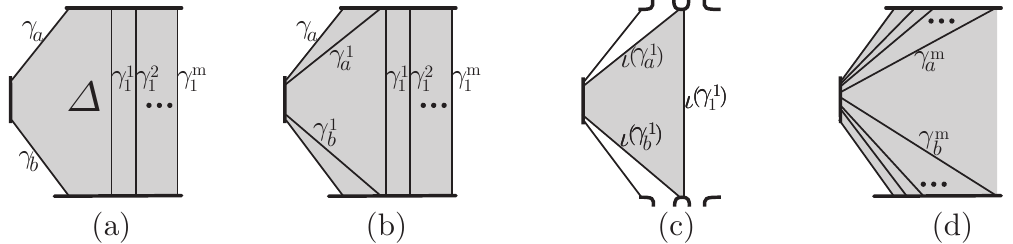}}
\caption[]{}
\label{fig_arcslide}
\end{figure}

\end{proof}

Finally:

\begin{cor}\label{cor_independence}
Let $\obd$ and $(\pg',\mon')$ be open book decompositions supporting some common contact structure on $M^3$. Then if there exists some basis $\B$ of $\pg$ such that, for any curve system $L$, $\B$ stably detects overtwistedness relative to $L$, then the same is true of $(\pg',\mon')$.

\end{cor}

\begin{proof}
By Giroux (Theorem \ref{thm_giroux}), it is sufficient to show that our property is preserved under stabilization and destabilization. In particular, it is sufficient to consider the case that $(\pg',\mon')$ is obtained from $\obd$ by a single stabilization/destabilization. 

For the case of a stabilization, let $\obd \leadsto (\pg',\tau_s\mon)$ be a stabilization via arc $\sigma$, and $L'$ a curve system in $\pg'$ (Figure \ref{fig_stab_2}(a)). Letting $\gamma$ denote the co-core of the stabilizing handle, we then let $L$ denote the result of `pinching' $L'$ by an isotopy supported in a neighborhood of $\gamma$ which contracts the maximal segment of $\gamma$ with endpoints in $L'$ (Figure \ref{fig_stab_2}(b)). We may then consider $L$ as a curve system in $\pg$ (Figure \ref{fig_stab_2}(c)).

\begin{figure}[h!]
\centering \scalebox{1.9}{\includegraphics{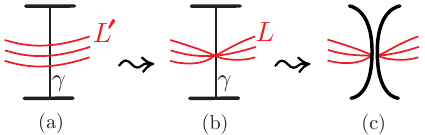}}
\caption[]{}
\label{fig_stab_2}
\end{figure}

Taking a handlebody decomposition of $\pg$ in which a neighborhood of $\sigma$ is the unique 0-handle, the set $\C$ consisting of the co-cores of the 1-handles gives a basis for $\pg$. Moreover, $\B := \mon^{-1}(\C)$ is another basis and has the property that $\mon(\B) \cap \sigma = \emptyset$. By Corollary \ref{cor_basis}, $\B$ stably detects overtwistedness relative to $L$, so we can find arc collection $\Gamma$, and stabilization sequences $S$ and $S'$ as in Definition \ref{def_detect}. We may of course assume $\Gamma$ lies in any neighborhood of $\B$, so that $\mon(\Gamma) \cap \sigma = \emptyset$, and whenever $\gamma_i$ and $\gamma_j$ are parallel arcs in $\Gamma$, they are again parallel in $\pg'$. But then we may apply the same stabilization sequence to $(\pg',\tau_s\mon)$, and (in the stabilized page $\pg''$) have $S\tau_s\mon(\Gamma) = S\mon(\Gamma)$. In particular then, as $(\pg'',S\mon, \Gamma)$ has an overtwisted region proper with respect to $S'(L)$, $(\pg'',S\tau_s\mon, \Gamma)$ has (the same) overtwisted region, proper with respect to $S'(L')$. Thus $\B \cup \{\gamma\}$ is a basis of $\pg'$ detecting overtwistedness relative to $L'$.

On the other hand, for a destabilization, we may of course choose the basis $\B$ of $\pg$ to contain the co-core $\gamma$ of the destabilization handle. Then letting $\B'$ denote the result of sliding an endpoint of $\gamma$ over some other element of $\B$, by Corollary \ref{cor_basis}, $\B'$ stably detects overtwistedness relative to $L$. Moreover, in $\pg'$, each element of $\B'$ is isotopic to an element of the basis $\B \setminus \{\gamma\}$. Thus any collection $\Gamma$ in $\pg$ satisfying the conditions of Definition \ref{def_detect} for $\B'$ is the image of the obvious inclusion of another such collection in $\pg'$. Finally, any curve system in $\pg'$ is again a curve system in $\pg$, so we are done.

\end{proof}

\section{Inconsistency, tightness, and Legendrian surgery}\label{sec_proofs}

We are now in a position to prove Theorems \ref{thm_tight} and \ref{thm_surgery}. 

\subsection{Consistency and tightness}

We start with (a slightly strengthened version of):

\begin{thm_1}
Let $M$ be a closed oriented 3-manifold, and $\xi$ a positive, co-oriented contact structure on $M$. Then the following are equivalent:

\begin{enumerate}

\item $\xi$ is overtwisted.

\item Some open book decomposition supporting $(M,\xi)$ is inconsistent.

\item Each open book decomposition supporting $(M,\xi)$ is inconsistent.

\item For any open book decomposition $\obd$ supporting $\xi$, and any basis $\B$ and curve system $L$ in $\pg$, $\B$ stably detects overtwistedness relative to $L$.

\item There exists an open book decomposition $\obd$ supporting $\xi$, and basis $\B$, such that for any curve system $L$ in $\pg$, $\B$ stably detects overtwistedness relative to $L$.

\end{enumerate} 

\end{thm_1}

\begin{proof}

In light of Giroux's classification theorem, equivalence of (4) and (5) is of course a restatement of Corollaries \ref{cor_basis} and \ref{cor_independence}. On the other hand (4) trivially implies (3), which in turn trivially implies (2). We will show then that (2) implies (1), and (1) implies (5). 

To start with, we generalize a construction due to Goodman (\cite{go}) to demonstrate an overtwisted disc in $(M,\xi)$ whenever $(M,\xi)$ is supported by an inconsistent open book decomposition. In particular, let $\obd$ be a supporting open book, with overtwisted region $A$, supported by minimal $\Gamma = \{\gamma_1,\gamma_2,\ldots,\gamma_n\}$, which we index such that $\gamma_i \cap \mon(\gamma_j)$ is a corner for $A$ for $i<n$ and $j = i+1$, or $i=n$ and $j=1$ (in fact one may assume $n=1$ - see Lemma \ref{lem_destab} - but we prefer to go through the construction for the general case). We then consider the suspension $S_i$ of $\gamma_i$ in the mapping torus of $\obd$, which we extend over the binding by attaching a meridional disc along each $\{p\} \times S^1$, for $p \in \partial \gamma_i$. We thus have a collection of embedded discs (one for each element of $\Gamma$) which we label $D_i$, and $n$ positive boundary-intersection points $\partial D_{i<n} \cap \partial D_{i+1}$ and $\partial D_n \cap \partial D_1$. We then `resolve' $\cup_iD_i$ at each intersection point $p$ by adding a pair of small triangles from the page $\pg_0$ through $p$ in the unique way which preserves the boundary orientation (Figure \ref{fig_suspension}). Smoothing the result via an isotopy relative to the boundary, and then pushing each $\partial \gamma_i$ into $\pg$, we obtain an embedded annulus $C \looparrowright M$, such that the boundary is contained in $\pg$, and exactly one boundary component of $C$ bounds a disc (namely, our original region $A$) in $\pg$. Finally, then, `capping off' a boundary component of $C$ via this disc, we obtain an embedded disc $D$ in $M$; following the arguments of Goodman, which in turn rely on the `Legendrian realization principle' of Honda \cite{ho1}, we may make $\partial D$ Legendrian. Moreover, the Thurston-Bennequin number of $\partial D$ is just given by the intersection of $D$ with a push-off of the boundary along $\pg$, so is zero. We conclude that $D$ is an overtwisted disc. 

\begin{figure}[h!]
\centering \scalebox{1.8}{\includegraphics{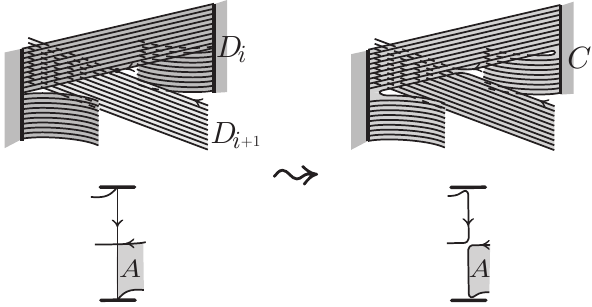}}
\caption[]{To the left: above, the discs $D_i$ and $D_{i+1}$, in a neighborhood of $\gamma_i$, illustrating the foliation of the mapping torus by the pages; below, the restriction to $\pg_0$. To the right, the result $C$ of resolving the intersections.}
\label{fig_suspension}
\end{figure}

For the final implication ($(1) \Rightarrow (5)$), suppose that $\xi$ is overtwisted. Using Eliashberg's homotopy classification of overtwisted contact structures \cite{el} it is straightforward to find an open book decomposition $\obd$ supporting $\xi$ which is a negative stabilization (i.e. replace the Dehn twist in Definition \ref{def_stab} with its inverse) of some other open book (see \cite{go} or \cite{hkm} for a proof). Let $\gamma$ denote the co-core of the stabilizing 1-handle, $L$ a collection of curves and arcs, and $U(\gamma)$ a neighborhood of $\gamma$ disjoint from $\partial L$ (Figure \ref{fig_stab}(a)). Orienting $\gamma$ arbitrarily, let $\sigma_1$ denote a boundary-parallel arc in a neighborhood of $\partial^-\gamma$ which intersects $\gamma$ exactly once (Figure \ref{fig_stab}(a)), $\sigma_2$ an arc isotopic to $\gamma$, isotoped in $U(\gamma)$ to lie to the right of $\gamma$, and not intersect $\sigma_1$ (Figure \ref{fig_stab}(b)), and $\sigma_3$ an arc isotopic to $\mon(\gamma)$, isotoped relative to its intersection with $\sigma_1$ by pushing the endpoints against the positive orientation of $\partial \pg$ in a neighborhood of $\partial \gamma$ disjoint from $\partial\sigma_1$ and $\partial\sigma_2$ (Figure \ref{fig_stab}(c)). It is then straightforward to check (Figure \ref{fig_stab}(d)) that after the associated stabilizations, our bigon is preserved, and is proper with respect to $L$. In particular then any basis of $\pg$ containing $\gamma$ stably detects overtwistedness relative to $L$.

\begin{figure}[h!]
\centering \scalebox{1.9}{\includegraphics{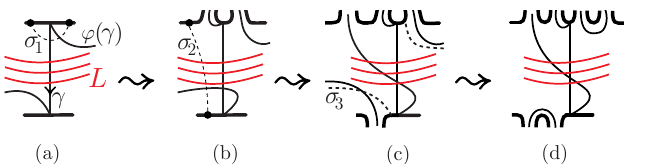}}
\caption[]{}
\label{fig_stab}
\end{figure}

\end{proof}

\subsection{Tightness and Legendrian surgery}

We gather for reference a pair of simple observations:

\begin{lem}\label{lem_destab}
Let $A$ be an overtwisted region in $\triple$, proper with respect to some curve system $L$. If $A$ is not a bigon, then there is $\gamma \in \Gamma$, and triple $(\pg',\mon',\Gamma \setminus \{\gamma\})$ which is obtained by a destabilization of $\triple$, and contains overtwisted $A'$ with two fewer sides than $A$, again proper with respect to $L$.

\end{lem}

\begin{proof}
Assuming $L \cap \Gamma$ is non-empty, by Definition \ref{def_proper} there is some negative corner $y$ of $A$ such that all intersections $L \cap (\Gamma \cup \mon(\Gamma))$ can be assumed to occur in any neighborhood of $y$. Let $\gamma$ then denote the element of $\Gamma$ encountered first traveling around $\partial A$ from $y$ against the positive orientation. Now from the definition of an overtwisted region it is clear that $\gamma$ is the co-core of a stabilization 1-handle; destabilizing, we see (Figure \ref{fig_destab}) that the effect of the destabilization on $A \cup \Gamma \cup \mon(\Gamma)$ can be realized as a resolution of the negative corner of $A$ in $\gamma$, followed by an isotopy to push the result away from $\partial \pg$ near $\partial \gamma$ (as in the proof of Theorem \ref{thm_tight}). The new region $A'$ then clearly has all desired properties.

\begin{figure}[h!]
\centering \scalebox{1.9}{\includegraphics{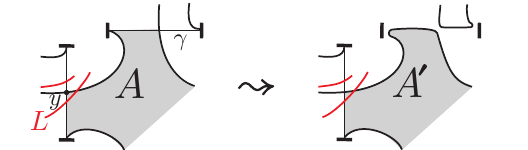}}
\caption[]{}
\label{fig_destab}
\end{figure}
\end{proof}

\begin{lem}\label{lem_braid}
Let $\obd \leadsto (\pg',\tau_s \mon)$ be a stabilization of an open book decomposition, and $L$ a simple closed curve in $\pg$. Then each of $(\pg',\tau^{-1}_{\tau_s(L)}\tau_s \mon)$ and $(\pg',\tau^{-1}_L \tau_s \mon)$ are stabilizations of $(\pg,\tau^{-1}_L \mon)$.
\end{lem}

\begin{proof}
Each of the statements follows easily from the well-known (and easily verified) fact that, for a given diffeomorphism $\psi$ of a surface $\pg$, and simple closed curve $L$ in $\pg$, we have $\psi \tau_L \psi^{-1} = \tau_{\psi(L)}$. The first statement indeed follows directly, so consider the second: re-factoring $\tau^{-1}_L\tau_s$ as $\tau_{\tau^{-1}_L s} \tau^{-1}_L$, observe that, as $L$ does not cross the stabilization 1-handle, $\tau^{-1}_L s$ again crosses it exactly once. 
\end{proof}

\begin{thm_2}\label{thm_legendrian}
If $(M,\xi)$ is obtained by Legendrian surgery on tight $(M',\xi')$, for $M'$ a closed, oriented 3-manifold, and $\xi'$ a positive co-oriented contact structure, then $(M,\xi)$ is tight. 
\end{thm_2}

\begin{proof}
Letting $L$ denote the Legendrian knot along which the surgery is performed, it follows from work of Giroux (see e.g. \cite{e}) that we may find an open book $\obd$ supporting $(M,\xi)$ such that $L$ is a curve on a page, and furthermore that $(M',\xi')$ is supported by $(\pg,\tau_L^{-1} \circ \mon)$. Suppose then that $(M,\xi)$ were overtwisted. In light of (the slightly strengthened version of) Theorem \ref{thm_tight}, we may find an collection $\Gamma$ in $\pg$, and a sequence of stabilizations $S$, such that the stabilized triple $(\pg', S\mon,\iota(\Gamma))$ has an overtwisted region $A$, and further a subsequence $S'$ of $S$ such that $\iota(\Gamma) \cap S' (L)$ is contained in a single element of $\iota(\Gamma)$. Using Lemma \ref{lem_destab}, we may destabilize each remaining element of $\Gamma$ to obtain an open book $(\pg'',\mon'')$ in which $\gamma$ is (after a bigon-removing isotopy of $\mon''(\gamma)$) mapped to the left at an endpoint, and $S'(L)$ still lies on the page. In particular then $\gamma$ is again mapped to the left in $(\pg'',\tau^{-1}_{S'(L)} \mon'')$, so by \cite{hkm} supports an overtwisted contact structure. On the other hand, (using Lemma \ref{lem_braid}), $(\pg',\tau^{-1}_{S'(L)}S\mon)$ is a common stabilization of $(\pg'',\tau^{-1}_{S'(L)} \mon'')$ and $(\pg,\tau_L^{-1} \circ \mon)$, so that $\xi'$ is overtwisted, a contradiction.

\end{proof}
As an aside, we note that rather than appealing to \cite{hkm} in the above proof, we could just as well simply show directly that $(\pg', \tau^{-1}_{S'(L)} S\mon,\iota(\Gamma))$ stably has an overtwisted region; indeed that it has (using the terminology of Section \ref{sec_properties}) a boundary-based region is immediate, so it is left to show that this region may be made isolated through stabilizations. An explicit algorithm to do just this can be found in \cite{w3}.

\end{document}